\documentclass{amsart}

\usepackage{amssymb}
\usepackage[all]{xy}
\usepackage{tikz-cd}
\usepackage[T1]{fontenc}
\usepackage{xfrac}

\tikzset{ampersand replacement=\&}

\newcommand{\mack}{\mathcal{M}\!\textit{ack}} 
\newcommand{\QXmod}[1]{\QQ[#1]\text{-mod}} 
\newcommand{\Sp}{\mathcal{S}\!\textit{p}} 
\newcommand{\set}{\mathcal{S}\!\textit{et}} 
\newcommand{\Orb}{\textrm{Orb}}  
\newcommand{\Orbx}{\textrm{Orb}^{\times}} 
\newcommand{\Comm}{\textrm{Comm}}
\newcommand{\A}{\mathcal{A}}
\newcommand{\CommAG}{\Comm\,\A(G)}
\newcommand{\Ho}{\textrm{Ho}}  
\newcommand{\gr}{\textrm{gr}}   
\newcommand{\Ch}{\mathsf{Ch}} 
\newcommand{\GSp}{G\textrm{-} \mathcal{S}\!\textit{p} _{\mathbb{Q}}}  
\newcommand{\QXcdga}[1]{\QQ[#1]\text{-CDGA}}   
\newcommand{\naive}{\textrm{Naive}}
\newcommand{\CDGA}{\text{CDGA}}

\newcommand{\QQ}{\mathbb{Q}} 
\newcommand{\CC}{\mathbb{C}} 
\newcommand{\ZZ}{\mathbb{Z}} 

\newcommand{\KU}{\textit{KU}} 
\newcommand{\ku}{\textit{ku}} 
\newcommand{\RU}{\textit{RU}} 
\newcommand{\burnQ}{\underline{A}_{\QQ}} 
\newcommand{\constQ}{\underline{\QQ}} 
\newcommand{\repQ}{\underline{\RU}_{\QQ}} 
\newcommand{\linepi}{\underline{\pi}} 
\newcommand{\s}{\mathbb{S}_{\mathbb{Q}}} 
\newcommand{\HQ} {H\constQ} 

\newcommand{\res}[2]{\text{res}_{#1}^{#2}} 
\newcommand{\tr}[2]{\text{tr}_{#1}^{#2}} 

\newcommand{\firstq}{\varphi}
\newcommand{\secondq}{\psi}

\newcommand{\sma}{\wedge}

\DeclareMathOperator{\im}{Im} 

\newcommand{\abs}[1]{\lvert{#1}\rvert}
\newcommand{\inv}{{-1}}

\DeclareMathOperator{\Galgp}{Gal} 

\newcommand{\xto}{\xrightarrow} 

\usepackage[capitalise]{cleveref}

\newtheorem{thm}{Theorem}[section] 
\newtheorem{cor}{Corollary}[section]
\makeatletter\let\c@cor\c@thm\makeatother
\newtheorem{lemma}{Lemma}[section]
\makeatletter\let\c@lemma\c@thm\makeatother

\makeatletter\let\c@prop\c@thm\makeatother

\makeatletter\let\c@claim\c@thm\makeatother
\newtheorem{thrm}{Theorem}[section]
\makeatletter\let\c@thrm\c@thm\makeatother

\newtheorem*{unnumberedtheorem}{Theorem}  

\theoremstyle{definition}
\newtheorem{defn}{Definition}[section]
\makeatletter\let\c@defn\c@thm\makeatother
\newtheorem{strategy}{Strategy}[section]
\makeatletter\let\c@strategy\c@thm\makeatother

\theoremstyle{remark}
\newtheorem{rem}{Remark}[section]
\makeatletter\let\c@rem\c@thm\makeatother
\newtheorem{ex}{Example}[section]
\makeatletter\let\c@ex\c@thm\makeatother

\makeatletter
\let\c@equation\c@thm
\numberwithin{equation}{section}
\makeatother

\title[Naive-commutative structure on rational $K$-theory]{Naive-commutative structure on rational equivariant $K$-theory for abelian groups}

\author[Bohmann]{Anna Marie Bohmann}
\address[Bohmann]{Vanderbilt University}
\email{am.bohmann@vanderbilt.edu}

\author[Hazel]{Christy Hazel}
\address[Hazel]{University of Oregon}
\email{chazel@uoregon.edu}

\author[Ishak]{Jocelyne Ishak}
\address[Ishak]{Vanderbilt University}
\email{jocelyne.ishak@vanderbilt.edu}

\author[K\k{e}dziorek]{Magdalena K\k{e}dziorek}
\address[K\k{e}dziorek]{Radboud University Nijmegen}
\email{m.kedziorek@math.ru.nl}

\author[May]{Clover May}
\address[May]{University of California Los Angeles}
\email{clovermay@math.ucla.edu}

\begin{document}

\begin{abstract}In this paper, we calculate the image of the connective and periodic rational equivariant complex $K$-theory spectrum in the algebraic model for naive-commutative ring $G$-spectra given by Barnes, Greenlees and K\k{e}dziorek for finite abelian $G$.  Our calculations show that these spectra are unique as naive-commutative ring spectra in the sense that they are determined up to weak equivalence by their homotopy groups. We further deduce a structure theorem for module spectra over rational equivariant complex $K$-theory.
\end{abstract}

\maketitle

\section{Introduction}

Modeling rational spectra via algebraic data has a long and fruitful history in homotopy theory.  Serre's original calculations of stable homotopy groups of spheres \cite{Serre1951} imply that the rational homotopy category $\Ho(\Sp_\QQ)$ is equivalent to the category of graded rational vector spaces.  An analogous equivalence was later obtained at the level of derived categories by Robinson \cite{Robinson} and at the level of model categories by Shipley \cite{ShipleyHZ} as a zig-zag of symmetric monoidal Quillen equivalences
\[\Sp_{\QQ} \simeq_Q \Ch (\QQ\textrm{-mod}).\]
Work of Richter and Shipley \cite{RichterShipley} further shows that there is a zig-zag of Quillen equivalences between rational commutative ring spectra and commutative differential graded algebras over $\QQ$.  Hence rational CDGAs are an \emph{algebraic model} for the rational commutative ring spectra.

It is something of a truism in algebraic topology that ``algebra is easy,'' in the sense that once one can reduce a topological question to a matter of algebra, the remaining algebraic computations should be straightforward.  Like most truisms, this one is mostly false: algebraic computations come equipped with a plethora of subtleties.  Moreover, the abstract knowledge that one can reduce a problem to algebra is often quite separate from the explicit reduction in a given case. In particular, for a concrete rational commutative ring spectrum $X$, it may be nontrivial to find the explicit rational CDGA corresponding to $X$ under the Richter--Shipley zig-zag of Quillen equivalences.   Nevertheless, algebraic models of homotopy theory---and more general algebraicizations of topological questions---are of great utility in both structural and computational understanding of homotopy theory.

In this paper, we focus on specific, concrete computations in algebraic models for rational $G$-equivariant spectra over a finite group $G$.  That is, our main goal is to find explicit models for rational $G$-spectra in the algebraic categories modeling these spectra.  The main spectra of interest are commutative ring spectra.

 For any finite group $G$, there is a model for the homotopy category of rational $G$-spectra given by work of Greenlees and May \cite{GM_Tate}. What we call an \emph{algebraic model} in this paper is not the model for the homotopy category of rational $G$-spectra, but an algebraic model category that is Quillen equivalent to the category of spectra in question.  In the case of $G$-spectra, \cite{KedziorekExceptional} uses Greenlees and May's result to produce an algebraic model category $\A(G)_\QQ$ that is Quillen equivalent to the stable model category of rational orthogonal $G$-spectra.

In the nonequivariant case, Richter and Shipley's result says that commutative algebra objects in the algebraic model for rational spectra are a model for rational commutative algebra spectra.  In the equivariant case,  the story is more intricate.  There is a hierarchy of types of ``equivariant commutativity'' \cite{BlumbergHill}, and commutative algebra objects in $\A(G)_\QQ$ only model the \emph{lowest} level of this commutativity, which is sometimes referred to as ``naive commutative''   \cite{BarnesGreenleesKedziorek}.
We denote this algebraic model for rational naive-commutative ring $G$-spectra by $\CommAG_\QQ$.  We provide a detailed description of the algebraic models $\A(G)_\QQ$ and $\CommAG_\QQ$ in Section \ref{sec:reviewrationalmodels}.

 Our main theorem is as follows.  It appears later as Theorem \ref{thm:main}.
\begin{unnumberedtheorem}
Let $G$ be a finite abelian group.  The image of $\KU_\QQ^G$ in the algebraic model
 $\CommAG_\QQ$ is given by $(V_H)_{(H)\leq G}$ where
\begin{itemize}
\item $V_H=0$ if $H$ is not cyclic and
\item when $H$ is cyclic of order $n$, $V_H\cong\QQ(\zeta_n)[\beta^{\pm 1}]$ where $\QQ(\zeta_n)$ is the field extension of $\QQ$ by a primitive $n$-th root of unity $\zeta_n$ and $\beta$ is in degree $2$.
\end{itemize}
\end{unnumberedtheorem}

Finding this image of $\KU_\QQ^G$ in the algebraic model is not simply a matter of tracing through the various functors in zig-zag of Quillen equivalences between rational naive-commutative ring $G$-spectra and $\CommAG_\QQ$.  This zig-zag includes functors for which we do not have explicit computational control.  What the zig-zag retains is control over the homology of the image in the algebraic model of a given spectrum $X$; in general this does not suffice to determine the algebraic object itself.  Hence the strategy of proof is to compute the homotopy groups of the geometric fixed points of $\KU_\QQ^G$.  These homotopy groups encode the homology of the algebraic model of $\KU_\QQ^G$ as a naive commutative ring $G$-spectrum.
We then show any commutative differential graded algebra with this homology is formal. The formality result finally determines the image of $\KU_\QQ^G$ in the algebraic model.

Our main result has several consequences. Firstly, it shows that all modules over $\KU_\QQ^G$ are free over the idempotent pieces of $\KU_\QQ^G$. This result is stated as Corollary \ref{cor:modules}.  In fact, our calculations show this holds for both abelian and nonabelian groups.
\begin{unnumberedtheorem}
Let $G$ be a finite group and $X$ be a module spectrum over $\KU_\QQ^G$. Then
\[ X\simeq \bigoplus_{(H)} (e_{(H)}\KU_\QQ^G)^{\oplus i_H} \oplus (\Sigma e_{(H)}\KU_\QQ^G)^{\oplus j_H},
\]
where  $H \leq G$,  and $i_H$ and $j_H$ are nonnegative integers.
\end{unnumberedtheorem}

Our formality result also shows that $\KU_\QQ^G$ admits a unique naive-commutative $E_\infty$ structure.  This result is stated as \cref{uniquenaiveEinftystructure}.
\begin{unnumberedtheorem}
Let $G$ be a finite abelian group.  Then $\KU_\QQ^G$ and $\ku_\QQ^G$ admit unique structures as naive commutative $G$-ring spectra, i.e.,~ as naive $E_\infty$-algebras in $G$-spectra.  That is, if $X$ is a rational naive-commutative $G$-ring spectrum whose graded Green functor of homotopy groups is isomorphic to that of $\KU_\QQ^G$ or $\ku_\QQ^G$, then there is a weak equivalence of rational naive-commutative $G$-ring spectra between $X$ and $\KU_\QQ^G$ or $\ku_\QQ^G$, respectively.
\end{unnumberedtheorem}

The computations in this paper also set the stage for an analysis for $\KU_\QQ^G$ as a genuine commutative ring spectrum.  This analysis, which uses recent work of Wimmer \cite{Wimmer}, is the subject of forthcoming work by the authors.

\subsection{Notation}
Throughout the paper we assume that $G$ is a finite group.
We use the notation $\A(G)_\QQ$ for the algebraic model of rational $G$-spectra and $\CommAG_\QQ$ for the algebraic model of rational naive-commutative ring $G$-spectra; see Definition \ref{def:1} and Definition \ref{def:2}, respectively.
We use the notation $\simeq_Q$ to denote a zig-zag of Quillen equivalences between model categories.
If $X$ is a rational naive-commutative ring $G$-spectrum then we denote by $\theta(X)$ its derived image in $\CommAG_\QQ$.

\subsection{Acknowledgments}
We would like to thank Hausdorff Research Institute for Mathematics in Bonn for their hospitality in hosting the Women in Topology III workshop, where much of this research was carried out.   We also extend our thanks to the organizers of the workshop, Julie Bergner, Ang\'elica Osorno, and Sarah Whitehouse, for making the event both possible and productive.  Many thanks are due to Brooke Shipley, who helped shape the initial stages of this research.  Warm thanks also to Dan Dugger and Mike Hill for  many helpful conversations.

In addition to funding for the workshop from the Hausdorff Institute, we are grateful to Foundation Compositio Mathematica and to the National Science Foundation of the United States, both of which provided funding for the workshop. NSF support was via the grants NSF DMS 1901795 and NSF HRD 1500481:\ AWM ADVANCE. Additionally, the fourth author was supported by a NWO Veni grant 639.031.757.  The first author was partially supported by NSF Grant DMS-1710534.

\section{Review of Rational Models}\label{sec:reviewrationalmodels}

In this section, we recall the construction of algebraic models in the world of  rational equivariant stable homotopy theory. We begin by reviewing the story at the level of homotopy groups and homotopy categories, followed by a discussion of the more structured story at the level of algebraic models via Quillen equivalences.

Given any $G$-spectrum $X$ and any integer $n$, the collection of homotopy groups $\{ \pi_n^H(X)\mid H \le G \}$ forms what is called a \emph{Mackey functor}.  The description of Mackey functors we follow is due to Dress \cite{Dress1973}. For an introduction to the theory of Mackey functors we refer the reader to \cite{Webb_guide}, \cite{shulman2014equivariant}, or \cite[\S 3.1]{HHR}.  When $X$ is a rational $G$-spectrum, its Mackey functor of homotopy groups $\linepi_n(X)$ is a rational Mackey functor, meaning for every subgroup $H$ of $G$, $\pi_n^H(X)$ is a rational vector space.  For example, if $X$ is the rational equivariant sphere spectrum $\s$, then the homotopy groups Mackey functor $\linepi_0(\s)$ is the rational Burnside ring Mackey functor $\burnQ$, which is defined by
\[\burnQ(G/H):= A(H) \otimes \QQ,\]
where $A(H)$ is the Burnside ring of $H$, i.e.~the Grothendieck ring of finite $H$-sets.  All higher homotopy groups of $\s$ vanish.

\begin{rem} The Burnside ring Mackey functor $\burnQ$ has more structure than simply that of a Mackey functor.  It is a commutative Green functor, which reflects the fact that $\s$ is a (naive) commutative $G$-spectrum.  In fact, $\burnQ$ has the even richer structure of a Tambara functor, although we will not make use of it in this paper.
\end{rem}

In this rational setting, Greenlees and May \cite{GM_Tate} show that the algebraic structure of Mackey functor homotopy groups determine the homotopy category of spectra in the following sense: they produce an equivalence of categories
\[ \Ho(\GSp) \to \gr(\mack(G)_\QQ)\]
from the homotopy category of rational $G$ spectra to the category of graded rational Mackey functors that is given by taking homotopy groups.  We note here that this functor is not induced by a Quillen equivalence of model categories.  Once in the algebraic setting of Mackey functors, idempotents in the Burnside ring allow a further splitting of rational Mackey functors into families of modules over group rings $\QQ[W_GH]$, where $W_GH$ is the Weyl group of a subgroup $H$ of $G$.   That is, Greenlees and May prove the following theorem.
\begin{thrm}[Greenlees--May \cite{GM_Tate}] \label{GM_appendix}
Idempotent splitting produces an equivalence of categories
\[
	 \gr(\mack_\QQ G ) \to \prod_{(H)\leq G}  \gr(\QXmod{W_GH}).
	\]
where $W_GH=N_GH/H$ is the Weyl group of $H$ as a subgroup of $G$.  There is thus an equivalence of categories
\[ \Ho(\GSp)\to \prod_{(H)\leq G}  \gr(\QXmod{W_GH}).\]
\end{thrm}
At each level of the grading, the functor from $\mack_\QQ(G)$ to $\prod\QXmod{W_GH}$ is given by sending a Mackey functor $M$ to the $W_GH$-module $e_HM(G/H)$ where $e_H$ is an idempotent in the rational Burnside ring for $G$ associated to the subgroup $H$.  We discuss these idempotents in more detail below in \cref{sec:idempotents_splittings}.

We are interested in incorporating additional structure that is not present in the homotopy category of rational $G$-spectra.  First, we wish to work at the model categorical level.
Theorem \ref{GM_appendix}'s splitting of the homotopy category of rational $G$-spectra for finite $G$ is mimicked to give a zig-zag of symmetric monoidal Quillen equivalences of monoidal model categories in \cite{KedziorekExceptional}
\begin{equation}\label{equation:alg_model}
 \GSp \simeq_Q \prod_{(H)\leq G}\Ch(\QQ[W_GH]\textrm{-mod}),
\end{equation}
where the product is over conjugacy classes of subgroups of $G$.  The model category $\prod_{(H)\leq G}\Ch(\QQ[W_GH]\textrm{-mod})$ has the objectwise projective model structure.
\begin{defn}\label{def:1} The model category
\[\prod_{(H)\leq G} \Ch(\QXmod{W_GH})\]
 is called the \emph{algebraic model for rational $G$-spectra} and is denoted $\A(G)_\QQ$.
\end{defn}
Note that the monoidal structure on $\A(G)_\QQ$ is given by tensor product over $\QQ$ in every product factor. One of the consequences of this result is that the derived image of the unit is the unit. That is, the sphere spectrum is sent to the constant sequence $\QQ$ concentrated in degree $0$ with trivial Weyl group actions.

Next we consider commutative ring structures on spectra.  This consideration is more subtle than in the non-equivariant case.  $G$-spectra have  a hierarchy of levels of ``equivariant commutativity'' \cite{BlumbergHill}.  Ring $G$-spectra with the \emph{lowest} level of commutativity are called \emph{naive}-commutative. Naive-commutative ring $G$-spectra are algebras for a $G$-operad equipped with a trivial $G$-action which is underlying $E_\infty$ when one forgets the $G$-action. An example of such a $G$-operad is the linear isometries operad on a trivial $G$-universe.

Barnes, Greenlees and K\k{e}dziorek \cite{BarnesGreenleesKedziorek} showed that commutative algebras in the algebraic model $\A(G)_\QQ$ for rational $G$-spectra model these \emph{naive-commutative} ring $G$-spectra.  That is, there is a zig-zag of Quillen equivalences
\begin{equation}\label{equivgivingnaivecommutativemodel}
\Comm_{\naive}(G\Sp_\QQ) \simeq_Q \Comm(\prod_{(H)\leq G}\Ch(\QQ[W_GH]\textrm{-mod})).
\end{equation}
\begin{defn}\label{def:2}
The model category $\Comm(\prod_{(H)\leq G}\Ch(\QXmod{W_GH}))$ is denoted by $\CommAG_\QQ$ and it has weak equivalences and fibrations created in $\A(G)_\QQ$. It is the \emph{algebraic model for rational naive-commutative ring $G$-spectra.}
\end{defn}

\begin{rem}
 Note that the product of these commutative differential graded algebras is equivalent to a diagram category
 \[
 \CommAG_\QQ= \prod_{(H)\leq G}   \QXcdga{W_GH}  \cong  \Orbx_G \textrm{-CDGA}_{\QQ},
 \]
Here  the category $\Orb_G$ is the orbit category spanned by transitive $G$-sets $G/H$ for $H \leq G$, and the morphisms are given by the set of $G$-equivariant maps. The category $\Orbx_G$ is the full subcategory of $\Orb_G$ consisting of isomorphisms.
\end{rem}

The image of a (naive-commutative ring) $G$-spectrum in the algebraic model is not very explicit, as the Quillen equivalences of (\ref{equation:alg_model}) and (\ref{equivgivingnaivecommutativemodel}) used in establishing the algebraic model use Shipley's result \cite{ShipleyHZ} (and Richter--Shipley's result \cite{RichterShipley}, respectively), which is not computationally trackable.

Let $\theta(X)$ denote the image of a naive-commutative rational $G$-spectrum $X$ in the algebraic model  $\CommAG_\QQ$.   The algebraic splitting of the category of graded Mackey functors (or commutative Green functors) using idempotents of the rational Burnside ring $A(G)\otimes \QQ$ is compatible with splitting rational $G$-spectra using the idempotents by \cite[Appendix A]{GM_Tate}.  Hence, by \cite{BarnesGreenleesKedziorek}  we know that the homotopy groups of the geometric fixed points of a naive-commutative ring $G$-spectrum $X$ are isomorphic to the homology of $\theta(X)$. That is, for each conjugacy class of subgroups $(H)$, the homology of the chain complex $\theta(X)_{(H)}$ is given by the homotopy groups of the $H$-geometric fixed points $\Phi^H(X)$:
\begin{equation}\label{homologyofimagehomotopygroupsofgeomfp}
  H_*(\theta(X)_{(H)})= \pi_*(\Phi^H(X)).
\end{equation}
 In fact, using this observation we can calculate the homology of the image of $X$ in the algebraic model using the splitting of rational Mackey functors, since
\[\pi_*(\Phi^H(X))\cong e_H\linepi_*(X)(H)
\]
where $e_H$ is an idempotent element in the Burnside ring $A(G)\otimes \QQ$. This compatibility is shown in \cite{GM_Tate}.

A key to identifying the image of a spectrum $X$ in the rational model is therefore to  calculate the idempotent pieces of the Mackey functors $\linepi_*(X)$.
In the next section we concentrate on understanding the action of the Burnside ring Green functor on a given Mackey functor and the behavior of the algebraic idempotent splitting.

\section{Splitting Mackey functors via idempotents in the Burnside ring}\label{sec:idempotents_splittings}

In this section, we give an overview of the idempotent splitting of rational Mackey functors for a finite group with the goal of providing the context necessary for the calculations in \cref{sec:calculations}. As mentioned, these results originate in \cite[Appendix A]{GM_Tate}.  For more details on the action of the Burnside ring Green functor on a Mackey functor $X$ and modern account of the idempotent splitting see \cite{BarnesKedziorek}.   We review the construction of the idempotent splitting in enough detail to suggest the essential calculational result, \cref{quotienttransfers}.  This result is proved in \cite{Schwede}.

 Let $G$ be a finite group. For the remainder of the paper, we suppress the notation for rationalization and let $A(G)$ denote the rational Burnside ring for $G$.  Recall that if $X$ is a finite $G$-set, then $X$ decomposes into orbits $G\cdot {x_1},\dots, G\cdot {x_n}$ and for each $x_i$ the orbit $G\cdot {x_i}$ is isomorphic to $G/\text{Stab}(x_i)$. For subgroups $H$ and $K$ in $G$, the orbits $G/H$ and $G/K$ are isomorphic as $G$-sets if and only if $H$ is conjugate to $K$. Thus a basis for $A(G)$ is given by
\[\{[G/H]\mid(H)\leq G\},\]
 where $(H)\leq G$ is used to denote a conjugacy class of subgroups in $G$. We will abuse notation by writing $(H)$ for both the set of subgroups conjugate to $H$ and for a single representative of this conjugacy class. Note if $K$ is another subgroup of $G$, the notation $(H)\leq K$ indicates $H$ is subconjugate to $K$ by an element of $G$.

By tom Dieck's result \cite[5.6.4, 5.9.13]{tomdieck}, the ring map
	\[
	\Phi\colon A(G) \to \prod_{(H)\leq G} \QQ \text{\qquad defined by\qquad } [X]\mapsto (\abs{X^{H}})_{(H)}.
	\]
	is an isomorphism. Thus it can be used to find idempotents in the ring $A(G)$. Define $e_{J}$ to be the pre-image of the projection onto the $(J)$-th factor in the product. That is,
\[e_{J}=\Phi^{-1}((\delta_{J}(H))_{(H)}),\]
where
\[
	\delta_{J}(H) = \begin{cases}1, &\text{if } (J)=(H) \\ 0,& \text{otherwise.} \end{cases}
	\]

Let $M$ be a rational Mackey functor on a finite group $G$. We can define an action of the Burnside ring $A(G)$ on $M$ as follows. Let $X$ and $Y$ be finite $G$-sets, and let
\[\pi\colon Y\times X \to Y\]
denote the projection. The action of $[X] \in A(G)$ on $M(Y)$ is given by the composite
	\begin{equation}\label{actionofburnsideringonamackeyfunctor}
	M(Y)\xto{\pi^{*}}M(Y\times X) \xto{\pi_{*}} M(Y),
	\end{equation}
 as is shown, for example, in \cite{GM_Mackey} or \cite{shulman2014equivariant}.
One can check this action is through ring maps, and so using the description of the idempotents in terms of the additive basis, we can decompose the Mackey functor $M$ as
\begin{equation}\label{idempotentsplittingofamackeyfunctor}
	M\cong \bigoplus_{(H)\leq G} e_HM.
\end{equation}
This is a decomposition as Mackey functors.  To deduce \cref{GM_appendix},
Greenlees and May make a further essential reduction by showing that for any $H$, the Mackey functor $e_HM$ is freely generated by the $W_GH$-module $e_HM(G/H)$.  Indeed, the ungraded case of \cref{GM_appendix} is an equivalence
\[ \mack(G)_\QQ \to \prod_{(H)\leq G} \QXmod{W_GH}\]
given by sending a Mackey functor $M$ to the sequence of modules $(V_H)$ defined by  \[V_H=e_HM(G/H).\]
The Weyl group action on $V_H$ is the inherent $W_GH$-action on the value of the Mackey functor $e_HM$ at $G/H$.

In order to understand the idempotent pieces of a Mackey functor more concretely, we would like an explicit description of the elements $e_H\in A(G)$.
The formula for the idempotents in terms of the additive basis was first introduced by Gluck in \cite{gluck}.
\begin{lemma}[\cite{gluck}]\label{lem:idempotentformula}
Let $H$ be a subgroup of $G$, then $e_H \in A(G)$ is given by the formula
\[
e_H =  \sum_{K \leqslant H} \frac{\abs{K}}{\abs{N_G H}} \mu(K,H) G/K
\]
where $\mu(K,H) = \Sigma_i (-1)^i c_i$ for $c_i$ the number of
strictly increasing chains of subgroups from $K$ to $H$ of length $i$.
The length of a chain is one less than the number of subgroups involved
and $\mu(H,H)=1$ for all $H \leqslant G$.
\end{lemma}

Thus the action of the idempotent element $e_H$ on a Mackey functor $M$ can be calculated from the action of orbits $G/K$ on the groups $M(G/J)$ for subgroups  $K\leq H$.   From the description of the action of the Burnside ring (\ref{actionofburnsideringonamackeyfunctor}), we know that $G/K$ acts on $M(G/J)$ by a sum of composites $\tr{?}{J}\circ\res{?}{J}$ between $J$ and various groups that are subconjugate to both $J$ and $K$, but the coefficients of this sum are not transparent. Thus, in general, it is not that easy to calculate $e_H M(G/J)$, where $(H)\leq J$. However, in this paper we only need to compute $V_H= e_H M(G/H)$ as a  $\QQ [W_GH]$-module.  In \cite{Schwede}, Schwede gives an inductive argument on the lattice of subgroups to obtain the following elegant description.

\begin{lemma}\label{quotienttransfers} \cite[Theorem 3.4.22]{Schwede} For $H\leq G$,
\[
V_H \cong M(G/H) / t_HM,
\]
where $t_HM$ is the subgroup of $M(G/H)$ generated by transfers from proper subgroups of $H$. Note that the Weyl group action on the Mackey functor descends to a $W_GH$-action on the quotient $V_H$ because of the compatibility axiom
\[
c_{g,K} \tr{H}{K}= \tr{gH}{gK} c_{g,H},
\]
where $H\leq K$, $g \in G$ and $c_{g,H}$ is the conjugation map
\[
c_{g,H}\colon M(G/H) \to M(G/ gHg^{-1}).
\]
\end{lemma}

\begin{rem} Observe that Maschke's theorem \cite[Proposition 1.5]{FultonHarris} applies to show that since $t_HM$ is a $\QQ[W_GH]$-module of $M(G/H)$, it has a complement and thus the quotient $M(G/H)/t_HM$ is in fact a direct summand of $M(G/H)$. This is of course necessary if $M(G/H)/t_HM$ is to be the  direct summand $e_HM(G/H)$.
\end{rem}

Lemma \ref{quotienttransfers} provides a tool at the heart of our strategy for computing the image $\theta(X)$ of a rational $G$-spectrum in the algebraic model.
For reference, we describe this strategy explicitly.  This is the procedure we employ in the next two sections to calculate the image of $\KU_\QQ^G$.
\begin{strategy}\label{strategyforcalculatingmodel} Let $X$ be a naive-commutative rational $G$-spectrum. A general strategy for attempting to calculate $\theta(X)$ is to do the following for a representative $H$ of each conjugacy class of subgroups of $G$:
\begin{enumerate}
\item Use \cref{quotienttransfers} to calculate $V_H=e_H\linepi_*(X)(G/H)$, together with its graded algebra structure.
\item Show that $V_H$ is formal as a commutative differential graded $\QQ[W_GH]$-algebra.
\item Use (\ref{homologyofimagehomotopygroupsofgeomfp}) to conclude that the  $(H)$-coordinate of $\theta(X)$ is weakly equivalent to $V_H$.
\end{enumerate}
These steps imply that each component of $\theta(X)$ is weakly equivalent to $V_H$.  Since the model category $\CommAG_\QQ$ is a product, we obtain a weak equivalence
\[\theta(X)\simeq (V_H).\]
\end{strategy}

We begin by illustrating this strategy on two simple examples, the Eilenberg--MacLane spectrum for the constant commutative Green functor $\constQ$ and the Eilenberg--MacLane spectrum for the rational Burnside Green functor $\burnQ$.  In both cases, the formality of Step 2 is immediate and the focus is on calculating the idempotent pieces of the Green functors using \cref{quotienttransfers}.

  \begin{ex} Let $G$ be a finite group.   Suppose $\constQ$ is the constant Green functor with value $\QQ$, i.e.  the value at any orbit $G/H$ is given by
 \[\constQ (G/H) = \QQ, \]   where the action of the Weyl group $W_GH$ is trivial. For $K \leq H\leq G$,
all restriction and conjugation maps are the identities and the transfer maps are given by \begin{align*}
\tr{K}{H}\colon \QQ & \rightarrow \QQ \\
          x & \mapsto \sum_{\gamma \in  W_HK} \gamma\cdot x=\sum_{\gamma \in  W_HK}x=\abs{W_HK} x.
\end{align*}
 Hence, for any subgroup $H$, the image of the transfer from the trivial subgroup $e$ is
\[ \Im(\tr{e}{H})= \QQ. \]
Thus the homology of the image of the equivariant rational Eilenberg--MacLane spectrum $\HQ$ in the algebraic model is
\[ V_H= \QQ / t_H\constQ=0,\]
for all subgroups $H \leq G$ except for the trivial subgroup, where
\[ V_e = \constQ (G/e)= \QQ . \]
Notice that $V_e=\QQ$ is concentrated in degree $0$ and is formal as a $\CDGA$. Thus the image of $\HQ$ in the algebraic model is weakly equivalent to the sequence of CDGAs with value $\QQ$ at the trivial group and zero at other groups.
 \end{ex}

\begin{ex}\label{sphere example}
Let $G$ be a finite group. The rational Burnside Mackey functor $\burnQ$ is the representable functor $\mathcal{B}^{op}_G(-,G/G)\otimes \QQ$ where $\mathcal{B}_G$ is the Burnside category. One can check that this is isomorphic to the Mackey functor whose value on the orbit $G/H$ is given by $A(H)$, the rational Burnside ring of $H$. For $H\leq K\leq G$, the restriction maps are given by
\[ \res{H}{K}([Y])=[i_H^{\ast}(Y)] , \]
 where $Y$  is a $K$-set, and
 \[
 i_H^{\ast}\colon \set^K \to \set^H
 \] is the forgetful functor from the category $\set^K$ of finite $K$-sets to the category of finite $H$-sets.
 The transfer maps are given by induction, i.e.
\[ \tr{H}{K}([X])=[K\times_H X], \]
where  $K\times_H X$ is the quotient of $K\times X$ given by $(kh,x)\sim(k,hx)$ for all $h\in H$. Note that $K$ acts on the left coordinate of the set $K\times_H X$.

Recall that the rational equivariant sphere spectrum $S_\QQ$ is the Eilenberg--MacLane spectrum $H\burnQ$ for the Burnside ring Mackey functor.   Using \cref{quotienttransfers} we can calculate the image of $S_\QQ$ in the algebraic model $\CommAG_\QQ$ by calculating the image of the transfers.

For each subgroup $H$, $\burnQ(G/H)=A(H)$ has an orbit basis given by
\[\{[H/J] \mid (J)_H\leq H\}.\]
Here it is important that we consider conjugation by $H$ instead of conjugation by $G$, and in general the conjugacy class $(J)_H$ may contain strictly fewer subgroups than $(J)_G$.   Let $J$ be a proper subgroup of $H$.  Observe that the image of the element $[J/e]\in A(J)$ under the transfer map $\tr{J}{H}\colon A(J)\to A(H)$ is
\[ [H\times_J J/e]=[H/J].\]
Hence all basis elements of the form $[H/J]$ for proper $J\leq H$ are in the image of the transfer.  Moreover, explicit calculation shows that no fixed $H$-set is in the image of a transfer map $\tr{J}{H}\colon A(J)\to A(H)$. Therefore, for each conjugacy class $(H)$, we have an isomorphism
\[V_H=A(H)/t_H\burnQ\cong \QQ\{[H/H]\},\]
concentrated in degree zero.  The $W_GH$-action on $A(H)$ is via conjugation, and is hence trivial on the basis element $[H/H]$; thus $V_H$ has a trivial $W_GH$-action.

This determines the homology of $\theta(\s)$ in $\CommAG_\QQ$. Since for each $(H)$, $\QQ$ concentrated in degree zero is formal as an object of $\QXcdga{W_GH}$, we find that $\theta(S_\QQ)$ is weakly equivalent to the constant sequence of $\QQ$'s with trivial Weyl group actions.
\end{ex}

\begin{rem} In fact, the image $\theta(\s)$ can be deduced from the construction of the zig-zag of Quillen equivalences in  \cite{KedziorekExceptional}.  This is a zig-zag of (symmetric) monoidal Quillen equivalences and thus, as mentioned in \cref{sec:reviewrationalmodels}, it sends the unit $\s$ in rational $G$-ring spectra to the unit in $\A(G)_\QQ$. Since the zig-zag of Quillen equivalences for naive-commutative rational ring $G$-spectra from \cite{BarnesGreenleesKedziorek} is a lift of the zig-zag of \cite{KedziorekExceptional}, the statement follows.  The calculation in \cref{sphere example} is presented as an illustration of the computational techniques on a familiar example.
\end{rem}

\section{Rational representation rings and Bott periodicity: homology level calculations}\label{sec:calculations}

Our main goal is to calculate the image of the ring spectrum $\KU^G_\QQ$ in the algebraic model $\CommAG_\QQ$ by implementing \cref{strategyforcalculatingmodel}.  Thus the first step towards understanding the image $\theta(\KU_\QQ^G)$ is to calculate the homotopy Mackey functors $\linepi_*(\KU_\QQ^G)$ and their idempotent splittings via the techniques of \cref{sec:idempotents_splittings}.  As in Display (\ref{homologyofimagehomotopygroupsofgeomfp}), the result of these calculations is the homology of $\theta(\KU_\QQ^G)$, which is recorded as \cref{lem:homology_of_image}.

Recall that $\linepi_0\KU^G_\QQ\cong \repQ^G$ where $\repQ^G$ is the rationalized representation ring Mackey functor. That is, the value of $\repQ^G$ at an orbit $G/H$ is the rationalization of the Grothendieck ring of complex $H$-representations $RU(H)$, the restriction maps are given by the restriction of representations, and the transfer maps are given by the induction of representations. The action of $W_GH$ on $RU(H)$ is given by $g\cdot[V]=[V_g]$ where $V_g$ is the $H$-representation such that $h\cdot v = (ghg^{-1})v$. We begin by studying the Eilenberg--MacLane spectrum for this Mackey functor. In order to describe the action of the Weyl group on the homology, we first define the following function.

\begin{defn}\label{defnofactionofweylgroupfunction}
Let $H$ be a cyclic subgroup of $G$ of order $n$ with a chosen generator $g$. Let $m_H$ denote the function $m_H\colon W_GH\to \ZZ/n$ where $m_H(a)\in \{1,\dots, n\}$ is such that $a^{-1}ga=g^{m_H(a)}$.
\end{defn}
In fact, $m_H\colon W_GH\to \ZZ/n$ is a homomorphism into the (multiplicative) group  of units $(\ZZ/n)^\times \subset \ZZ/n$.
\begin{lemma} The map $m_H\colon W_GH\to \ZZ/n$ is a group homomorphism $W_GH\to (\ZZ/n)^\times$.
\end{lemma}
\begin{proof}
It is straightforward to check that $m_H(ab)=m_H(a)m_H(b)$ and if $e\in W_GH$ is the identity, $m_H(e)=1$; moreover $m_H(a^\inv)$ is clearly the inverse to $m_H(a)$ so that the image of $m_H$ is contained in the units.
\end{proof}

We can now state the homology of the image of $H\repQ^G$ in the algebraic model.

\begin{lemma}\label{homologyofrepnringmackeyfunctor}
The homology of $\theta(H\repQ^G)$, the image of $H\repQ^G$ in the algebraic model, is given by $(V_H)_{(H)\leq G}$ where $V_H=0$ when $H$ is not cyclic and when $H$ is cyclic of order $n$, $V_H\cong\QQ(\zeta_n)$, where $\QQ(\zeta_n)$ is the field extension of $\QQ$ by a primitive $n$-th root of unity $\zeta_n$. The action of $a\in W_GH$ on $\QQ(\zeta_n)$ is given by $a\cdot \zeta_n = \zeta_n^{m_H(a)}$.  That is, the action of $W_GH$ is given via the homomorphism $m_H\colon W_GH\to (\ZZ/n)^\times\cong \Galgp(\QQ(\zeta_n)/\QQ)$.
\end{lemma}
\begin{proof}
Let $H$ be a subgroup of $G$ and consider the map
	\[
	\bigoplus_{C\leq H} \text{ind}_C^H\colon \bigoplus_{C\leq H} RU(C)\otimes \QQ \to RU(H) \otimes \QQ
	\]
where $C$ runs over all cyclic subgroups of $H$. By a theorem of Artin, this map is surjective (see \cite[9.2.17]{Serre1977}, for example). By Lemma \ref{quotienttransfers}, the module $V_H$ is found by quotienting the image of all transfers of proper subgroups, so we immediately see $V_H=0$ if $H$ is not cyclic.

Now suppose $H$ is a cyclic subgroup of order $n$. We first show $V_H\cong \QQ(\zeta_n)$ as a $\QQ$-algebra.
Fix a generator $g$ for $H$. For each divisor $d$ of $n$, there is one subgroup of $H$ of order $d$, and the fixed generator for this subgroup will be $g^{n/d}$. Denote this subgroup by $H_d$. Note that each subgroup $H_d$ gives rise to a unique conjugacy class of subgroups in $G$.

In what follows, fix a primitive $n$-th root of unity $\zeta_n$ and let $\zeta_d$ be the primitive $d$-th root of unity $\zeta_n^{n/d}$. The complex representation ring for the cyclic group of order $d$ is isomorphic to $\QQ[x_d]/(x_d^d-1)$ where $x_d$ corresponds to the one-dimensional irreducible representation such that the generator $g^{n/d}$ acts via multiplication by $\zeta_d$. Here we use the subscript $d$ to keep track of which subgroup we are considering. The choices made in the previous paragraph show the restriction maps are given by $\res{H_{d_1}}{H_{d_2}}(x_{d_2})=x_{d_1}$ for two divisors of $n$ such that $d_1\mid d_2$.

The polynomial $x_d^d-1$ factors as a product of cyclotomic polynomials $\Phi_j$ where $j\mid d$. By the Chinese remainder theorem, the map
	\begin{align*}
	\QQ[x_d]/(x_d^d-1)&\to \prod_{j\mid d} \QQ[x_d]/(\Phi_j(x_d))\\
	f(x)&\mapsto (f(x), \dots, f(x))
	\end{align*}
is an isomorphism. We will use this interpretation throughout our computation.

We need to compute the image of the various transfer maps in
	\[
	RU(H)\otimes \QQ \cong \prod_{j\mid n}\QQ[x_n]/(\Phi_j(x_n))\cong \prod_{j\mid n}\QQ(\zeta_j).
	\]
For a divisor $d$, let's start by finding $\tr{H_d}{H}(1)$. The unit is given by the one-dimensional trivial representation, and inducing the trivial $H_d$-representation gives the $H$-representation $\mathbb{C}[H/H_d]$. Using character theory, one can check
	\[
	[\mathbb{C}[H/H_d]]=1+x_n^{d}+x_n^{2d}+\dots+x_n^{(\sfrac{n}{d}-1)d}.
	\]
We can factor this as a product of cyclotomic polynomials by observing
	\[
	x^n_n-1=(x_n^d-1)(1+x_n^{d}+x_n^{2d}+\dots+x_n^{(\sfrac{n}{d}-1)d}), ~\text{and so}
	\]
	\[
	\tr{H_d}{H}(1)=1+x_n^{d}+x_n^{2d}+\dots+x_n^{(\sfrac{n}{d}-1)d}=\prod_{j \nmid d} \Phi_j(x_n).
	\]
Thus the transfer of $1$ is zero in all factors indexed by divisors of $n$ that do not divide $d$. For the divisors of $d$, the above shows the $j$-th factor is given by
\[
(\tr{H_d}{H}(1))_j=(\sfrac{n}{d}-1)d+1=n-d+1
\]
because $x_n^d=(x_n^j)^{\sfrac{d}{j}}=1$ in this factor. To summarize, let $\delta_{j,d}$ be defined by $\delta_{j,d}=1$ if $j\mid d$ and $\delta_{j,d}=0$ if $j\nmid d$. We have shown
	\[
	(\tr{H_d}{H}(1))_j=\delta_{j,d}\cdot (n-d+1).
	\]
By Frobenius reciprocity, $\tr{H_d}{H}(x_d^\ell)=\tr{H_d}{H}(\res{H_d}{H}(x_n^\ell)\cdot 1) = \tr{H_d}{H}(1)\cdot x^\ell_n$, and so
	\[
	(\tr{H_d}{H}(x_d^\ell))_j=\delta_{j,d}(n-d+1)\cdot\zeta_j^\ell.
	\]
We conclude the image of the transfer is given by
	\[
	\im(\tr{H_d}{H})= \text{Span}\{(\delta_{j,d}\zeta_j^\ell)_{j\mid n} \mid \ell=0,\dots, d-1 \}.
	\]

To find $V_H$, we need to quotient by $\im(\tr{H_d}{H})$ for all divisors $d$ of $n$ such that $d\neq n$. We can do this inductively beginning with $d=1$, and the above shows that everything will be killed except for the factor
	\[
	\QQ[x_n]/(\Phi_n(x_n)) = \QQ(\zeta_n).
	\]
Thus $V_H\cong \QQ(\zeta_n)$.

Next we determine the action of the Weyl group. Observe the action of $a \in W_GH$ on $H$ is given by $a\cdot g = a^{-1}ga=g^{m_H(a)}$. The action on $x_n\in RU(H)$ is determined as follows. The class $x_n$ is represented by the representation $V$ that is a one-dimensional complex vector space such that $g$ acts via multiplication by $\zeta_n$. The twisted representation $a\cdot V = V_a$ has the same underlying vector space as $V$, but the action of $g$ is given by first conjugating by $a$. Thus $g$ acts in $V_a$ as $a^{-1}ga$ acts in $V$. Hence the action of $g$ on $V_a$ is given by multiplication by $\zeta_n^{m_H(a)}$ and $a\cdot x_n=x_n^{m_H(a)}$. From the proof above, we see this corresponds in the quotient $V_H$ to $a\cdot \zeta_n = \zeta_n^{m_H(a)}$.
\end{proof}

\begin{rem}\label{rem:trivAction}
If $G$ is abelian, then the conjugation action of the Weyl group using the function $m_H$ of \cref{defnofactionofweylgroupfunction} is trivial.  Thus the above analysis implies that the action of $W_GH$ on $\QQ(\zeta_n)$ is trivial.
\end{rem}

Using equivariant Bott periodicity (see for example \cite{Segal} or \cite[XIV.3]{AlaskaNotes}), we extend \cref{homologyofrepnringmackeyfunctor} to a result about $\KU_\QQ^G$.

\begin{lemma}\label{lem:homology_of_image}
The homology of $\theta(\KU_\QQ^G)$, the image of $\KU_\QQ^G$ in the algebraic model, is given by $V_H=0$ when $H$ is not cyclic and $V_H=\QQ(\zeta_n)[\beta^{\pm 1}]$ with $\abs{\beta}=2$ when $H$ is a cyclic group of order $n$. The action of $a\in W_GH$ on $\QQ(\zeta_n)[\beta^{\pm 1}]$ is given by $a\cdot \zeta_n = \zeta_n^{m_H(a)}$ and $a\cdot \beta=\beta$.
\end{lemma}
\begin{proof}

We begin by reviewing the Mackey functor structure of $\linepi_*\KU_\QQ^G$. Recall $\linepi_0(\KU_\QQ^G)\cong \repQ^G$. In fact, $\linepi_0(\KU^G)\cong \RU^G$ before rationalizing. For any complex representation $V$ and any finite pointed $G$-$CW$ complex $X$, Bott periodicity provides a natural isomorphism
	\[
	[X,\KU^G]_G \cong [S^V\sma X,\KU^G]_G
	\]
which is given by multiplication by the Bott class $\beta_V$ \cite[{Theorem XIV.3.2}]{AlaskaNotes}. Let $\beta$ denote $\beta_{\CC}$ where $\CC$ is the one-dimensional trivial representation. 

If $X=G/H$, then the Bott periodicity isomorphism shows
	\[
	\KU^{j}_G(G/H)=[G/H_+, \Sigma^j\KU^G]_G \cong [S^{2n}\sma G/H_+, \Sigma^j\KU^G]_G = \KU^{-2n+j}_G(G/H)
	\]
where the isomorphism is given by multiplying by $\beta^n$. For any $G$-spectrum $E$, the coefficient Mackey functor $\underline{E}^{-*}(\mathit{pt})$ is isomorphic to the homotopy group Mackey functor $\linepi_{*}(E)$. Thus on the homotopy group level,
	\[
	\linepi_j(\KU^G)(G/H)\cong \linepi_{j+2n}(\KU^G)(G/H).
	\]
When $j$ is odd, note $\linepi_{1}(\KU^G)(G/H)=0$ from the comments in \cite[{Section XIV.3}]{AlaskaNotes}, and thus by periodicity, all homotopy groups in odd degrees are zero. We have determined that as a graded ring, $\linepi_*(\KU^G)(G/H)\cong RU(H)[\beta^{\pm 1}]$ where $\abs{\beta}=2$. For clarity of the proof, we will decorate the polynomial generator $\beta\in \linepi_*(\KU^G)(G/H)$ with a subscript $H$ to keep track of which level of the Mackey functor it lives in.

We next determine the restriction, transfer, and Weyl group action on the elements $\beta_H$. For $H\leq K$, the restriction map is induced by the quotient map $G/H \to G/K$, and so the naturality of the Bott class implies the following diagram commutes
\[
\xymatrix{
\KU^0_G(G/K)\ar[r]^{\cdot\beta_K} \ar[d]_{\res{H}{K}}& \KU^{2}_G(G/K)\ar[d]_{\res{H}{K}}\\
\KU^0_G(G/H)\ar[r]^{\cdot\beta_H} & \KU^{n}_G(G/H)
}
\]
Consider the image of $1\in \KU^0_G(G/K)$. Going around the diagram the two different ways will show $\res{H}{K}(\beta_K)=\beta_H$ in $\linepi_*(\KU_G)$. The transfer map is also induced by a stable map of orbits and $\tr{H}{K}(1)=\abs{W_KH}$, so we have $\tr{H}{K}(\beta_H)=\abs{W_KH}\beta_K$.

To find the action of the Weyl group, note the action by $a\in W_GH$ is induced by the map $a\colon G/H\to G/H$, $eH\mapsto aH$. By taking $H=K$ in the diagram above and replacing restriction by the map induced by $a$, we can consider the image of the unit and use that $a\cdot 1 = 1$ to see $a\cdot \beta_H=\beta_H$.

We now return to the algebraic model. The value of $V_H$ as a $\QQ$-algebra follows readily from \cref{homologyofrepnringmackeyfunctor} and the periodicity shown above. The action of $W_GH$ on $\QQ(\zeta_n)$ is the same as that of \cref{homologyofrepnringmackeyfunctor}, and the action on $\beta$ is trivial since it was trivial in the original Mackey functor.
\end{proof}

Since the homotopy groups of idempotent pieces of $\KU_\QQ^G$ are a graded field, we obtain the following result.

\begin{cor}\label{cor:modules}
Let $G$ be a finite group and $X$ be a module spectrum over $\KU_\QQ^G$. Then
\[ X\simeq \bigoplus_{(H)} (e_{(H)}\KU_\QQ^G)^{\oplus i_H}\oplus (e_{(H)}\KU_\QQ^G)^{\oplus j_H},
\]
where  $H \leq G$, and  $i_H$ and $j_H$ are nonnegative integers.
\end{cor}
\begin{proof}
Let $X$ be a module spectrum over $\KU_\QQ^G$.  Then $X$ is determined by its image in $\CommAG_\QQ$, which is determined by the idempotent pieces $e_{(H)}(X)$.  Since $X$ is a module over $\KU_\QQ^G$, Barnes's splitting result \cite{barnessplitting} implies that for each conjugacy class of subgroups $(H)$, $e_{(H)}X$ is a module over $e_{(H)}\KU_\QQ^G$. The homotopy groups of $e_{(H)}\KU_\QQ^G$ are a graded field, so any module spectrum over $e_{(H)}\KU_\QQ^G$ is free, and is thus a wedge of suspensions of $e_{(H)}\KU_\QQ^G$. Since  $e_{(H)}\KU_\QQ^G$ is $2$-periodic, it's enough to consider the suspensions $e_{(H)}\KU_\QQ^G$ and $\Sigma e_{(H)}\KU_\QQ^G$.
\end{proof}

Let $\ku_\QQ^G$ denote the connective cover of $\KU_\QQ^G$.  Bott periodicity also provides the homology of the image of $\ku_\QQ^G$ in the algebraic model.
\begin{lemma}\label{homologyimageconnective}
The homology of $\theta(\ku_\QQ^G)$, the image of $\ku_\QQ^G$ in the algebraic model, is given by $V_H=0$ if $H$ is not cyclic and $V_H=\QQ(\zeta_n)[\beta]$ with $\abs{\beta}=2$ if $H$ is a cyclic group of order $n$. The action of $a\in W_GH$ on $\QQ(\zeta_n)[\beta^{\pm 1}]$ is given by $a\cdot \zeta_n = \zeta_n^{m_H(a)}$ and $a\cdot \beta=\beta$.
\end{lemma}


\section{Formality: The image of $\KU_\QQ^G$ in the algebraic model}\label{sec:formality}

Our main goal is to find the image of $\KU_\QQ^G$ in the algebraic model $\CommAG_\QQ$ when $G$ is a finite abelian group. \cref{lem:homology_of_image} calculates the homology of the image $\theta(\KU_\QQ^G)$,  but in general, the homology of a rational $\CDGA$ is not enough to determine its isomorphism class in the homotopy category. In the abelian case, following \cref{strategyforcalculatingmodel}, we will show that $\theta(\KU_\QQ^G)$ is formal. That is, if $(A_\bullet)_{(H)}$ is an object of $\CommAG_\QQ$  such that $(H(A_\bullet))_{(H)}$ is isomorphic to $\pi(\phi^K\KU_\QQ^G)$ in the category $\prod \gr\QQ[W_GH]\text{-alg}$,
then there exists a zig-zag of quasi-isomorphisms of CDGAs from $(A_\bullet)_{(H)}$ to $(H(A_\bullet))_{(H)}$ where $(H(A_\bullet))_{(H)}$ is the tuple of chain complexes with zero differentials given by the homologies of the complexes $(A_\bullet)_(H)$. This will imply $(A_\bullet)_{(H)} \cong (H(A_\bullet))_{(H)}$ in the homotopy category $\Ho(\CommAG_\QQ)$.

In an effort to simplify the exposition, we prove the main formality result we want in several lemmas, which show formality of increasingly complicated $\QQ[W_GH]$-$\CDGA$s.  In the case of interest, while we know that the action of $W_GH$ is trivial on homology, we cannot assume that the action on the underlying chain complex itself is trivial.  We begin by looking at chain complexes in degree zero and add a free generator.

Our essential technique is to construct a single zig-zag of quasi-isomorphisms between a chain complex $A_\bullet$ and its homology $H(A_\bullet)$ of the form
\[\begin{tikzcd}
A_\bullet  \& D_\bullet \arrow{r}{\simeq} \arrow[swap]{l}{\simeq} \&  H(A_\bullet).
\end{tikzcd}
\]
where $D_\bullet$ is an appropriately ``free'' commutative differential graded algebra $D_\bullet \in \QXcdga{W_GH}$, given by a polynomial algebra tensored with an exterior algebra.  It is free in the sense that an algebra map out of $D_\bullet$ is determined by defining a chain map on the generators.  In the case where $H(A_\bullet)$ has trivial action, we can choose $D_\bullet$ to have trivial action.

As a warm-up and illustration of the construction, we consider the algebraic model for $\KU_\QQ$ nonequivariantly. This is the special case where $G$ is the trivial group.
\begin{lemma}\label{QBeta}
Suppose there exists $A_\bullet \in \CDGA_\QQ$ such that $H(A_\bullet) \cong \QQ[\beta^{\pm 1}]$.  Then $A_\bullet$ is formal, i.e.\ there is a zig-zag of quasi-isomorphism algebra maps
\begin{center}
\begin{tikzcd}
A_\bullet  \& D_\bullet \arrow{r}{\simeq} \arrow[swap]{l}{\simeq} \&  \QQ[\beta^{\pm 1}].
\end{tikzcd}
\end{center}
\end{lemma}

\begin{proof}
Let $\alpha \in A_2$ be an element representing $\beta$ so that $[\alpha] = \beta$.  It is not necessarily the case that $\alpha$ is invertible in $A_\bullet$, so we may not be able to define a map directly from $\QQ[\beta^{\pm 1}]$ to $A_\bullet$.  Instead, we will replace $\QQ[\beta^{\pm 1}]$ with a quasi-isomorphic free commutative differential graded algebra $D_\bullet$.
Let $\bar{\alpha} \in A_{-2}$ be a class that represents $[\beta^{-1}]$.  Although it is possible that $\alpha \bar{\alpha} \neq 1$, because $\alpha$ is invertible in homology there exists $\sigma \in A_1$ such that $d(\sigma) = 1 - \alpha \bar{\alpha}$.

Define the replacement differential graded algebra $D_\bullet$ to be the polynomial algebra on classes $\gamma$ and $\bar{\gamma}$ tensored the exterior algebra on a class $y$ with $\abs{\gamma} = 2$, $\abs{\bar{\gamma}} = -2$, and $\abs{y} = 1$, where $d(\gamma) = d(\bar{\gamma}) = 0$ and $d(y) = 1 - \gamma \bar{\gamma}$.  That is, $D_\bullet = \QQ[\gamma, \bar{\gamma}] \otimes E(y)$.  To define $\firstq\colon \QQ[\gamma, \bar{\gamma}] \otimes E(y) \to A_\bullet$
we only need to specify $\firstq(\gamma)$, $\firstq(\bar{\gamma})$, and $\firstq(y)$ such that $\firstq$ is a chain map.
We can define $\firstq(\gamma) = \alpha$, $\firstq(\bar{\gamma})=\bar{\alpha}$, and $\firstq(y) = \sigma$.  Notice that $\firstq$ is a quasi-isomorphism by construction.

Finally we can define a map $\secondq\colon \QQ[\gamma, \bar{\gamma}] \otimes E(y) \to \QQ[\beta^{\pm 1}]$ by $\secondq(\gamma) = \beta$, $\secondq(\bar{\gamma}) = \beta^{-1}$ and $\secondq(y) = 0$.  This map induces an isomorphism on homology and so we have a zig-zag of quasi-isomorphism algebra maps
\begin{center}
\begin{tikzcd}
A_\bullet \& \QQ[\gamma, \bar{\gamma}] \otimes E(y) \arrow[swap]{l}{\simeq} \arrow{r}{\simeq} \& \QQ[\beta^{\pm 1}]
\end{tikzcd}
\end{center}
and hence $A_\bullet$ is formal.
\end{proof}

As an immediate corollary, this shows that as an $E_\infty$ ring nonequivariant $\KU_\QQ$ is unique.  The uniqueness of $\KU_\QQ$, and indeed of $\KU$, as an $E_\infty$ ring was shown previously by Baker and Richter in \cite{BakerRichter} using obstruction theory.

Before turning to the formality arguments in the general context of $\QQ[W_GH]$-$\CDGA$s, we make the following useful observation.

\begin{lemma}[Averaging]\label{averagingtrick}
Let $A_\bullet$ be a $\QXcdga{W_GH}$, and suppose a homology class $x\in H(A_\bullet)$ is fixed under the $W_GH$-action.  Then $x$ has a representative $a \in A_\bullet$ that is fixed under the $W_GH$-action. Similarly if $y\in A_\bullet$ such that $d(y)$ is fixed under the $W_GH$-action, then there exists a fixed element $b$ such that $d(b)=d(y)$.
\end{lemma}
\begin{proof}
Choose an arbitrary cycle $a_0$ in $A_\bullet$ representing $x$. Then the element
 \[a = \mathrm{av}_{W_GH}(a_0):= \frac{1}{\abs{W_GH}}\sum_{g\in W_GH}g \cdot a_0
 \]
 is a cycle, it also represents $x$ in homology and is $W_GH$-fixed.

 For the class $y$, define $b=\mathrm{av}_{W_GH}(y)$. Since $d(y)$ is fixed and the differential is an equivariant map, $d(b)=d(y)$ and $b$ is fixed by construction.
\end{proof}

 In \cref{QZetaDegreeZero,RDegreeZero,RBeta} we often apply the averaging trick of \cref{averagingtrick} to chose fixed representatives for homology classes without further mention.

\begin{lemma}\label{QZetaDegreeZero}
Let $A_\bullet$ be a $\QQ[W_GH]$-$\CDGA$ such that its homology $H(A_\bullet)$ is isomorphic to $\QQ(\zeta_n)$ concentrated in degree zero and has trivial $W_GH$ action.  Then $A_\bullet$ is a formal as a $\QQ[W_GH]$-$\CDGA$.
\end{lemma}

\begin{proof}
We use a standard Koszul resolution. Let $D_\bullet \in \QXcdga{W_GH}$ be the free commutative differential graded algebra with trivial $W_GH$ action $\QQ[t] \otimes E(z)$ where $\abs{t} = 0$ and $\abs{z}=1$.  Let $d(t)=0$ and $d(z) = \Phi_n(t)$.  Notice that $z$ is in an odd degree, so graded commutativity implies $z^2 = 0$.  Thus the chain complex $D_\bullet$ is
\[
0 \to \QQ[t]\{z\} \to \QQ[t] \to 0
\]
with $D_0 = \QQ[t]$, $D_1 = \QQ[t]\{z\}$, and all other $D_k = 0$.  By construction $D_\bullet$ has homology $\QQ(\zeta_n)$ concentrated in degree zero with trivial action.

To define a map $\firstq\colon D_\bullet \to A_\bullet$, choose a class $a \in A_0$ that represents $\zeta_n$ so $[a] = \zeta_n$. Now $\zeta_n$ is a root of $\Phi_n(x)$ but it is possible that $a$ is not a root in $A_0$. However $\Phi_n(a)$ must be a boundary, so there exists a class $\rho$ such that $d(\rho) = \Phi_n(a)$. By \cref{averagingtrick}, we may assume $a$ is fixed  under $W_GH$. Since $a$ is fixed, the polynomial $\Phi_n(a)$ is fixed, and so we can also assume $\rho$ is fixed. Thus we may define $\firstq(t) = a$ and $\firstq(z) = \rho$. Notice the map $\firstq$ is now a quasi-isomorphism. 

We may easily define a quasi-isomorphism $\secondq\colon D_\bullet \to \QQ(\zeta_n)$ by $\secondq(a) = \zeta_n$ and $\secondq(z) =0$.  Thus we have a zig-zag of quasi-isomorphisms
\begin{center}
\begin{tikzcd}
A_\bullet \& D_\bullet \arrow[swap]{l}{\simeq} \arrow{r}{\simeq} \& \QQ(\zeta_n),
\end{tikzcd}
\end{center}
which completes the proof that $A_\bullet$ is formal.
\end{proof}

\cref{QZetaDegreeZero} applies to the homology of the representation ring Mackey functor, which shows that $\theta(H\repQ^G)$ is weakly equivalent to the homology specified in \cref{homologyofrepnringmackeyfunctor}.
\begin{cor}
For $G$ abelian, the image of $H\repQ^G$ is unique in $\CommAG_\QQ$.
\end{cor}

\begin{proof}
By Lemma \ref{RBeta} we have formality of the algebraic model $\theta(H\repQ^G)$ at each $(H)$.  The values of the algebraic model at each conjugacy class of subgroup are independent so the image of $H\repQ^G$ is unique up to equivariant quasi-isomorphism in $\CommAG_\QQ$.
\end{proof}

More generally, if $R$ is the quotient of a polynomial $\QQ$-algebra by a regular sequence, viewed as a $\QQ[W_GH]$-algebra with trivial action, then $R[\beta^{\pm 1}]$ is formal. We prove this in two steps.  The first generalizes \cref{QZetaDegreeZero}.

\begin{lemma}\label{RDegreeZero} Let $R$ be the quotient of a finitely generated polynomial algebra over $\QQ$ by a finite regular sequence and let $A_\bullet \in \QXcdga{W_GH}$ have homology $H(A_\bullet) \cong R$ concentrated in degree zero with trivial $W_GH$ action.  Then $A_\bullet$ is formal.
\end{lemma}

\begin{proof}
By assumption,  $R \cong \QQ[x_1, \dots, x_n]/I$ where $I$ is an ideal generated by a regular sequence of finitely many polynomials $I = (f_1, \dots, f_m)$.
Let $a_1, \dots, a_n \in A_0$ be $W_GH$-fixed elements representing the homology classes $x_1, \dots, x_n$ and choose $b_1, \dots, b_m \in A_1$ also $W_GH$-fixed such that $d(b_i) = f_i(a_1, \dots, a_n)$.  Define $D_\bullet \in \QXcdga{W_GH}$ to be the free $\CDGA$ with trivial action given by
\[D_\bullet = \QQ[t_1, \dots, t_n] \otimes E(z_1, \dots, z_m)\]
where $\abs{t_i} = 0$, $\abs{z_i}=1$, and $d(z_i) = f_i(t_1, \dots, t_n)$.  That is, $D_\bullet$ is the Koszul complex for the regular sequence $(f_1,\dots,f_m)$.  By the regularity of the sequence $(f_1,\dots,f_m)$, the homology  $H(D_\bullet)$ is isomorphic to $R$ concentrated in degree zero.

Now define $\firstq\colon D_\bullet \to A_\bullet$ by $\firstq(t_i) = a_i$ and $\firstq(z_i)=b_i$; again, this map is equivariant by our choices of $W_GH$-fixed representatives $a_i$ and $b_i$.  The map $\firstq$ is a quasi-isomorphism.  Define $\secondq\colon D_\bullet \to R$ by $\secondq(t_i) = x_i$ and $\secondq(z_i)=0$.  This is also a quasi-isomorphism.  Thus we have constructed a zig-zag of quasi-isomorphisms
\begin{center}
\begin{tikzcd}
A_\bullet \& D_\bullet \arrow[swap]{l}{\simeq} \arrow{r}{\simeq} \& R
\end{tikzcd}
\end{center}
and hence $A_\bullet$ is formal.
\end{proof}

Now we generalize the technique of Lemma \ref{QBeta} to incorporate the invertible class $\beta$.

\begin{lemma}\label{RBeta}
Let $A_\bullet \in \QXcdga{W_GH}$ with $H(A_\bullet) \cong R[\beta^{\pm 1}]$ where $R$ is a the quotient of a finitely generated polynomial algebra over $\QQ$ by a regular sequence, $\abs{\beta} = 2$ and where $R$ and $\beta$ have trivial $W_GH$-action.  Then $A_\bullet$ is formal.
\end{lemma}

\begin{proof}
 As in the proof of \cref{RDegreeZero}, let $D_\bullet\cong \QQ[t_1,\dots,t_n]\otimes E(z_1,\dots,z_m)$ be the Koszul complex so that $H(D_\bullet)\cong R$.  To adjoin the invertible class $\beta^{\pm 1}$, we tensor $D_\bullet$ with the chain complex $\QQ[\gamma,\bar\gamma]\otimes E(y)$ constructed in \cref{QBeta}. Since we are working over a field, the K\"unneth theorem for chain complexes implies that there are isomorphisms on homology
\[H(D_\bullet \otimes \QQ[\gamma,\bar{\gamma}]\otimes E(y))\cong H(D_\bullet)\otimes H(\QQ[\gamma,\bar{\gamma}]\otimes E(y))\cong R\otimes \QQ[\beta^{\pm 1}].\]
Let $\alpha \in A_2$ be a $W_GH$-fixed representative of $\beta$ and $\bar{\alpha} \in A_{-2}$ be a $W_GH$-fixed representative of $\beta^{-1}$.
It is possible that $\alpha \bar{\alpha} \neq 1$ in $A_\bullet$ but since $[\alpha][\bar{\alpha}] =1$ in homology, there exists a $W_GH$-fixed element $c \in A_1$ such that $d(c) = 1 - \alpha \bar{\alpha}$.
We extend $\firstq$ from the previous lemma to $\bar{\firstq}\colon D_\bullet \otimes \QQ[\gamma, \bar{\gamma}] \otimes E(y) \to A_\bullet$ via $\bar{\firstq}(\gamma) = \alpha$, $\bar{\firstq}(\bar{\gamma})=\bar{\alpha}$, and $\bar{\firstq}(y) = c$.  We also extend the map $\secondq$ from \cref{RDegreeZero} to
$\bar{\secondq}\colon D_\bullet \otimes \QQ[\gamma, \bar{\gamma}] \otimes E(y) \to R[\beta^{\pm 1}]$ via $\bar{\secondq}(\gamma) = \beta$, $\bar{\secondq}(\bar{\gamma}) = \beta^{-1}$ and $\bar{\secondq}(y)=0$.  Now $\bar{\firstq}$ and $\bar{\secondq}$ define a zig-zag of quasi-isomorphisms
\begin{center}
\begin{tikzcd}
A_\bullet \& D_\bullet \otimes \QQ[\gamma, \bar{\gamma}] \otimes E(y) \arrow[swap]{l}{\simeq} \arrow{r}{\simeq} \& R[\beta^{\pm 1}]
\end{tikzcd}
\end{center}
and hence $A_\bullet$ is formal.
\end{proof}

Notice this last lemma would also hold with $\abs{\beta} = 2n$, adjusting the degrees of $\gamma$ and $\bar{\gamma}$ appropriately.

As discussed in \cref{rem:trivAction}, when $G$ is an abelian group, the actions on the homology of $\theta(\KU_\QQ^G)$ are trivial.  Hence \cref{RBeta} applies to show that $\theta(\KU_\QQ^G)$ is formal for an abelian group $G$.
\begin{lemma}\label{lem:formality}
Let $G$ be a finite abelian group. Then $\theta(\KU_\QQ^G)$ is formal in $\CommAG_\QQ$.
\end{lemma}

\begin{proof} In Lemma \ref{RBeta} we have shown formality of the algebraic model $\theta(\KU_\QQ^G)$ at each conjugacy class of subgroups $(H)$.  As above, the values of the algebraic model at each conjugacy class of subgroup are independent so $\theta(\KU_\QQ^G)$ is formal.
\end{proof}

\begin{thrm}\label{thm:main} Let $G$ be a finite abelian group. The image of periodic $K$-theory $\KU_\QQ^G$ in the algebraic model is given by
$\QQ(\zeta_n)[\beta^{\pm 1}]$ with $\abs{\beta}=2$ with trivial action of the Weyl group for each cyclic subgroup $C_n \leq G$ and is zero for non-cyclic subgroups.
\end{thrm}

\begin{proof} This follows from Lemma \ref{lem:homology_of_image} and Lemma \ref{lem:formality}.
\end{proof}

A similar formality argument as above together with calculations from Section \ref{sec:calculations} show the following result.

\begin{thm}\label{thrm:connectivemain}
When $G$ is a finite abelian group, the image of connective $K$-theory $\ku_\QQ^G$ in the algebraic model is given by $\QQ(\zeta_n)[\beta]$ with $\abs{\beta} =2$ for each $C_n \leq G$ and zero otherwise. The action of the Weyl group $W_G(C_n)$ on $\QQ(\zeta_n)$ is trivial.
\end{thm}

\begin{cor}\label{uniquenaiveEinftystructure} $\KU_\QQ^G$ and $\ku_\QQ^G$ admit unique structures as naive commutative $G$-ring spectra, i.e.,~ as naive $E_\infty$-algebras in $G$-spectra.  That is, if $X$ is a rational naive-commutative $G$-ring spectra whose graded Green functor of homotopy groups is isomorphic to that of $\KU_\QQ^G$ or $\ku_\QQ^G$, then there is a weak equivalence of rational naive-commutative $G$-ring spectra between $X$ and $\KU_\QQ^G$ or $\ku_\QQ^G$ (respectively).
\end{cor}
\begin{proof}
For concreteness, suppose $X$ is a naive-commutative rational $G$-spectrum and the graded commutative homotopy group Green functor of $X$  is isomorphic to that of $\KU_\QQ^G$.  Let $\theta(X)$ be the image of $X$ in $\CommAG_\QQ$. Then the homology of $\theta(X)$ is isomorphic to the homology of $\theta(\KU_\QQ^G)$; since we know the latter to be formal, there is a zig-zag of quasi-isomorphisms
$\theta(X)\sim \theta(\KU_\QQ^G)$
in $\CommAG_\QQ$.  The zigzag of Quillen equivalences
\[ \Comm_\naive \GSp \simeq_Q \CommAG_\QQ\]
implies that $X$ and $\KU_\QQ^G$ are thus weakly equivalent in $\Comm_\naive\GSp$.  The proof for $\ku_\QQ^G$---or indeed, for any spectrum whose image we know to be formal in $\CommAG_\QQ$ ---is the same.
\end{proof}

\bibliography{wit3references}
\bibliographystyle{plain}

\end{document}